\documentclass[11pt]{article}
\usepackage[margin=1in]{geometry}
\usepackage{amsmath,amssymb,amsthm}
\usepackage{mathtools}
\usepackage{microtype}
\usepackage{booktabs,array,needspace}
\usepackage[hidelinks]{hyperref}
\usepackage{authblk}
\newtheorem{theorem}{Theorem}
\newtheorem{lemma}{Lemma}

\newtheorem{corollary}{Corollary}

\newtheorem{claim}{Claim}
\newtheorem{definition}{Definition}
\newcommand{\OPT}{\operatorname{OPT}}
\newcommand{\ind}{\mathbf 1}
\hypersetup{pdftitle={Linear Programming Upper Bounds for the Coil-in-the-Box Problem},
 pdfauthor={Kevin M. Byrnes, Eric T. Harley, Florin Spinu}}
\title{Linear Programming Upper Bounds\\for the Coil--in--the--Box Problem}
\author[1]{Kevin M. Byrnes}
\author[2]{Eric T. Harley}
\author[3]{Florin Spinu}

\affil[1]{Baltimore, MD, USA\\
\texttt{kevin.byrnes.jr@gmail.com}}

\affil[2]{South Kingstown, RI, USA\\
\texttt{eric.harley@gmail.com}}

\affil[3]{New York, NY, USA}

\date{}
\begin{document}
\maketitle

\begin{abstract}
Let $S_n$ denote the maximum length of a coil in the $n$-dimensional hypercube $Q_n$, i.e., a simple cycle whose induced subgraph has no additional chords. 
We combine classical square-counting identities and structural properties of maximum coils 
with new valid combinatorial inequalities into a parameterized linear program whose optimal value provides an upper bound for $S_n$.
For $n\ge13$, we construct two analytic dual feasible solutions. 
Their objective values yield upper bounds on $S_n$, and the smaller value is the exact optimum of the parameterized LP.
One dual solution results in the bound
\[
S_n\le
2^{n-1}
\frac{8n(n-1)(5n-14)}
{40n^3-131n^2+4n+136-(5n-14)\lfloor3n/4\rfloor},
\]
which improves upon the upper bounds in \cite{Lukito,LvZ,Zemor} for $13\le n\le19853$.
The second yields the upper bound
\[
S_n\le
2^{n-1}\left(
1-\frac1{30\sqrt{3n}}+O(n^{-1})
\right),
\]
improving upon the asymptotic upper bound of \cite{Zemor}.
Taken together, the two bounds improve upon the upper bounds for $S_n$ in \cite{Lukito,LvZ,Zemor} for every $n\ge13$.
\end{abstract}



\section{Introduction}\label{sec:intro}
The vertices of the $n$-dimensional hypercube $Q_n$ are the binary vectors in $\{0,1\}^n$, with two vertices adjacent when they differ in exactly one coordinate.
A \emph{coil} is an induced cycle in $Q_n$; equivalently, no two nonconsecutive vertices of the cycle are adjacent in $Q_n$.
Let $S_n$ denote the maximum length, i.e., the number of vertices, of a coil in $Q_n$.
The exact values of $S_n$ have been determined for $n\le 8$ \cite{OP, MutzeSurvey, OrlandEtAl}.

There is a long sequence of upper bounds on $S_n$.  The bound
\[
S_n \le 2^{n-1}\left(1-\frac{2}{n^2-n+2}\right),\qquad n\ge 7,
\]
was proved in \cite{Soloveva}.  The bound in \cite{Snevily} is
\[
S_n \le 2^{n-1}\left(1-\frac{1}{20n-41}\right),\qquad n\ge 12
\]
which was subsequently refined in \cite{Lukito} to
\begin{equation}
\label{eq:Lukito}
S_n \le 2^{n-1}\left( 1-\frac{1}{4n-9} \right) \text{ for } n\ge 11.
\end{equation}
The asymptotic upper bound established in \cite{Zemor} is
\begin{equation}
\label{eq:Zemor}
S_n\le 1+2^{n-1}\frac{6n}{6n+\sqrt n/(6\sqrt6)-7}
\le 2^{n-1}\left(1-\frac{1}{89\sqrt n}+O(n^{-1})\right)
\end{equation}
but the finite-dimensional bounds obtained from ``square-counting'' methods are stronger over a large range of intermediate dimensions.

The finite-dimensional bound in \cite{LvZ}, improving on \cite{EL}, is
\begin{equation}
\label{eq:LvZ}
S_n\le 2^{n-1}\left(1-\frac{2n-20}{5n^2-3n-20}\right)
+\frac{16n-80}{5n^2-3n-20},\qquad n\ge 7,
\end{equation}
and it was shown there to improve upon \eqref{eq:Zemor} for $7\le n\le19081$.
The bound \eqref{eq:Lukito} is superior for $11\le n \le 29$.

Let $U_L(n)$ and $U_{LvZ}(n)$ denote the right-hand sides of \eqref{eq:Lukito} and \eqref{eq:LvZ}, respectively, and let $U_Z(n)$ denote the
finite bound $1+2^{n-1}6n/(6n+\sqrt n/(6\sqrt6)-7)$ in \eqref{eq:Zemor}.
To the best of our knowledge, the upper bounds \eqref{eq:Lukito} and \eqref{eq:LvZ} are the strongest known for $11\le n \le 19081$,
while \eqref{eq:Zemor} is the strongest asymptotically among these earlier bounds \cite{MutzeSurvey}.

In this paper we make two contributions.
First, we derive new valid combinatorial inequalities for the structure of maximum coils in terms of the aggregate degree counts of vertices in the complement,
building upon the work of \cite{EL, Zemor}.
Second, we consider the continuous relaxations of these inequalities and analyze their interactions in a parametric linear program whose optimal value upper bounds $S_n$.
This is in the spirit of linear-programming bounds in coding theory and spherical codes \cite{Delsarte,DGS} and Euclidean sphere--packing \cite{CohnElkies}.

The data in the LP are ill-scaled for large values of $n$, but by analyzing the optimal solutions for small values of $n$ we form an ansatz and devise two \emph{analytic} dual feasible solutions whose objective function values provide explicit upper bound formulas for $S_n$.
One dual solution yields
\[
S_n\le
2^{n-1}
\frac{8n(n-1)(5n-14)}
{40n^3-131n^2+4n+136-(5n-14)\lfloor3n/4\rfloor}.
\]
This bound improves $U_L(n)$ and $U_{LvZ}(n)$ for every $n\ge13$, and $U_Z(n)$ for $13\le n\le19853$.
The second dual solution provides an upper bound of
\[
S_n\le
2^{n-1}\left(
1-\frac1{30\sqrt{3n}}+O(n^{-1})
\right),
\]
whose leading correction improves that obtained from the finite formula in \eqref{eq:Zemor} by a factor of $6\sqrt2/5$.
Together, the two solutions yield
\[
S_n<\min\{U_L(n),U_{LvZ}(n),U_Z(n)\}
\qquad(n\ge13).
\]
\noindent Finally, we establish that one of the two dual solutions must be optimal (depending on $n$ and the value of the LP parameter) and present a matching primal solution.

Section~\ref{sec:prelim} collects the supporting combinatorial lemmas and notation.
Section~\ref{sec:LP} presents the linear program, derives the primal candidates,
and proves their optimality by matching dual solutions. Section~\ref{sec:comparison} compares
the resulting upper bound with the classical bounds and gives its
asymptotic analysis. Section~\ref{sec:conclusions} concludes with questions concerning
integrality and combinatorial realizability.

\section{Supporting lemmas}\label{sec:prelim}
Throughout the paper $n\ge13$, unless explicitly stated otherwise.
Set
\[
N=2^{n-1},
\qquad
M=\left\lfloor\frac{3n}{4}\right\rfloor.
\]
Let $V(Q_n)$ denote the vertex set of $Q_n$, let $S\subseteq V(Q_n)$ be a maximum-length coil, $S_n = |S|$, and let $Y=V(Q_n)\setminus S$.  For $x\in V(Q_n)$, let $N(x)$ denote its neighborhood in $Q_n$.
For $x\in Y$ define
\[
s(x)=|N(x)\cap S|,\qquad Y_m=\{x\in Y:s(x)=m\},\qquad y_m:=|Y_m|.
\]
Thus $y_m$ is the number of vertices outside the coil having exactly $m$ neighbors on $S$.
For use in later inequalities, put
\[
\eta(y)=\sum_{m=0}^{M}y_m.
\]
This is the number of complement vertices having at most $M$ neighbors on the coil.
Every vertex of $S$ has exactly two neighbors in $S$ and $n-2$ in $Y$, hence
\begin{align}
\sum_{m=0}^n y_m&=2^n-S_n, \label{eq:Ycount}\\
\sum_{m=0}^n m y_m&=(n-2)S_n.\label{eq:SYedgecount}
\end{align}
From \cite[Lemma~2]{Snevily} we have the further equality
\begin{equation}
y_n=0.\label{eq:Yn}
\end{equation}

\subsection{Square-counting lemmas}\label{subsec:squares}
We classify the squares (4-cycles) of $Q_n$ as in \cite{EL}: 
$\mathcal{F}_1$ is the family of squares containing exactly three vertices of $S$; 
$\mathcal{F}_2$ is the family of squares containing exactly two adjacent vertices of $S$; 
$\mathcal{F}_3$ is the family of squares containing exactly two opposite vertices of $S$; 
$\mathcal{F}_4$ is the family of squares containing exactly one vertex of $S$; 
and $\mathcal{F}_5$ is the family of squares containing no vertex of $S$.  Since $S$ is an induced cycle and $S_n> 4$ in the range of $n$ considered here \cite{AK}, no square of $Q_n$ can have all four of its vertices in $S$; thus $\mathcal{F}_1$--$\mathcal{F}_5$ exhaust all possible squares.

For $1\le i\le5$, let $F_i=|\mathcal{F}_i|$.  We will need the following square identities:
\begin{align}
F_1&=S_n,\label{eq:F1}\\
F_1+2F_3&=\sum_m\binom m2 y_m,\label{eq:F1F3count}\\
F_4+4F_5&=\sum_m\binom{n-m}{2} y_m,\label{eq:F4F5count}\\
2F_3+F_4&=\binom{n-2}{2}S_n.\label{eq:Ssquare}
\end{align}
Here \eqref{eq:F1} is equation~(8) of \cite{EL}; \eqref{eq:F1F3count} and \eqref{eq:F4F5count} are equations~(10) and~(12) in Lemma~1 of \cite{EL}, after translating the complement degree there to $n-s(x)$; and \eqref{eq:Ssquare} is used in the proof of \cite[Lemma~5]{EL} and originates in \cite{Soloveva}.

We first develop two inequalities for the linear programming upper bound on $S_n$ in Section~\ref{sec:LP}.
The first is an inequality relating $y_{n-1}, y_{n-2}, y_{n-3}$ that arises from counting $\mathcal{F}_1$ square incidences using the argument of \cite[Lemma~4]{EL}.
Its proof is somewhat detailed and is deferred to Appendix~\ref{app:Yn2refinement}.

\begin{lemma}
\label{lem:Yn2refinement}
For $n\ge13$,
\begin{equation*}
(n-1)y_{n-1}+4y_{n-2}-y_{n-3}\le2S_n.
\end{equation*}
\end{lemma}

\begin{lemma}\label{lem:Alower}
For $n\ge13$,
\begin{equation*}
\sum_{m=0}^{n}\binom m2 y_m
\ge (n^2-3n+3)S_n-\frac{n(n-1)}4\,2^n.
\end{equation*}
\end{lemma}

\begin{proof}
From \eqref{eq:F1}, \eqref{eq:F1F3count}, and \eqref{eq:Ssquare},
\[
F_4=\left(\binom{n-2}{2}+1\right)S_n
-\sum_{m=0}^{n}\binom m2 y_m.
\]
Combining this with \eqref{eq:F4F5count} and using the identity
\[
\binom{n-m}{2}=\binom n2-(n-1)m+\binom m2
\]
together with \eqref{eq:Ycount} and \eqref{eq:SYedgecount}, one obtains
\begin{equation}
\label{eq:F5identity}
\sum_{m=0}^{n}\binom m2 y_m
=(n^2-3n+3)S_n-\frac{n(n-1)}4\,2^n+2F_5.
\end{equation}
The result follows from $F_5\ge0$.
\end{proof}

\subsection{Heavy-segment lemmas}\label{subsec:heavy}
For $x\in Y$, write
\[
\delta(x)=\deg_Y(x)=n-s(x).
\]
A complement vertex is called \emph{heavy} when $\delta(x)\ge n/4$.
Equivalently, it belongs to $Y_m$ for some $m\le M$, so
$\eta(y)$ counts the heavy complement vertices.

We use two consequences of the structural analysis in \cite{Zemor}.
For coils of length at least nine, \cite[Lemma~3.2]{Zemor} says that every segment of $S$ consisting of seven consecutive edges contains a vertex having a heavy neighbor in $Y$.
We call such a vertex of $S$ \emph{good}.
The lower bound in \cite{AK} gives $S_n>10$ in the present range, so this local result applies.

\begin{lemma}
\label{lem:Zprojection}
Let $S$ be a maximum-length coil in $Q_n$ having length $S_n$ and let $M=\left\lfloor\frac{3n}{4}\right\rfloor$.
Then
\begin{equation}\label{eq:Zcut}
\sum_{m=0}^{M}m y_m\ge \frac{S_n}{8}.
\end{equation}
\end{lemma}

\begin{proof}
Call a set of $k$ cyclically consecutive vertices of $S$ a $k$-segment and let $\mathcal{G}$ denote the set of all good vertices of $S$.
There are $S_n$ distinct $8$-segments of $S$, each of which contains at least one good vertex by \cite[Lemma~3.2]{Zemor}.
Since each $v\in \mathcal{G}$ is shared by 8 such segments, $S_n/8 \le |\mathcal{G}|$.

For each $v\in \mathcal{G}$, choose one neighbor $x\in Y$ satisfying $\deg_Y(x)\ge \frac{n}{4}$.
Call $x$ the \emph{associated} vertex and let $\mathcal X_{\mathrm g}$ denote the set of all associated vertices.
Observe that $x\in Y_m$ for some $m$; thus $n-m = \deg_Y(x) \ge \frac{n}{4}$, so $m\le M$.
Also note that each $x\in \mathcal X_{\mathrm g}$ can be associated with at most $m$ vertices in $\mathcal{G}$ (when $x\in Y_m$). 
Thus we have
\[
S_n/8 \le |\mathcal{G}| \le \sum_{m=0}^M m\left|\{x\in \mathcal X_{\mathrm g} \cap Y_m\}\right| \le \sum_{m=0}^M m y_m
\]
which yields \eqref{eq:Zcut}.
\end{proof}

We also use the \emph{heavy segments} introduced in \cite{Zemor}.
Addition of binary vectors below is componentwise modulo two.

\begin{definition}[Heavy segment]
\label{def:heavy}
A heavy segment is a triple $(\sigma,x,\nu)$ with the following properties.
The segment $\sigma$ consists of ten consecutive edges of $S$, with center vertex $O$;
the associated vertex $x\in Y$ is adjacent to $O$ and satisfies $\delta(x)\ge n/4$;
and $\nu=x+O$.
Moreover, there is a noncentral vertex $V$ of $\sigma$ such that $V+\nu\in S$.
\end{definition}

By \cite[Lemma~4.1]{Zemor}, there exists a family $\mathcal H$ of pairwise vertex-disjoint heavy segments such that
\begin{equation}
\label{eq:packing}
H:=|\mathcal H|\ge\frac{S_n-10}{18}.
\end{equation}
Let $\mathcal X$ be the set of distinct associated vertices, and set
\[
K=|\mathcal X|,
\qquad
h_x=|\{\sigma\in\mathcal H:x\text{ is associated with }\sigma\}|.
\]
Every associated vertex is heavy, so $K\le\eta(y)$.
Since $S_n>10$, \eqref{eq:packing} implies $H>0$ and hence $K>0$.
Thus
\begin{equation}
\label{eq:multiplicities}
\sum_{x\in\mathcal X}h_x=H,
\qquad
\sum_{x\in\mathcal X}h_x^2\ge\frac{H^2}{K}.
\end{equation}

For each heavy segment, choose as in \cite{Zemor} a closest noncentral vertex $X$ to its center $O$ with $X+\nu\in S$.
For two distinct heavy segments $\sigma,\sigma'$ with the same associated vertex, let $O,O'$ be their centers and $\nu,\nu'$ their corresponding labels.
The ordered pair $\pi=(\sigma,\sigma')$ is called \emph{admissible} when
\begin{align*}
V+\nu'&\in Y &&(V\in\sigma),\\
V'+\nu&\in Y &&(V'\in[O',X']),\\
d_S(O,O')&\ge13,
\end{align*}
where $[O',X']$ denotes the subpath of $S$ from $O'$ to $X'$ and $d_S$ denotes distance along $S$.
Let $\mathcal J_x$ be the admissible ordered pairs associated with $x$, and put
$J=\sum_{x\in\mathcal X}|\mathcal J_x|$.

The counting argument in \cite[Lemma~6.2]{Zemor} gives
\begin{equation}
\label{eq:paircount}
J\ge\sum_{x\in\mathcal X}h_x(h_x-16).
\end{equation}
Moreover, \cite[Lemma~5.3]{Zemor} associates to each admissible pair
$\pi=(\sigma,\sigma')$ four distinct vertices
$u_1(\pi),\ldots,u_4(\pi)\in Y$ such that
\begin{equation}
\label{eq:outputroles}
u_i(\pi)+\nu'\in\sigma\quad(i=1,2,3),
\qquad
u_4(\pi)+\nu\in[O',X'],
\qquad
\sum_{i=1}^4\delta(u_i(\pi))\ge n.
\end{equation}

For completeness, we make explicit the four-output construction of \cite[Lemma~5.3]{Zemor}.
Write $X'=V'_k$, where $1\le k\le5$, and let $e'_1,\ldots,e'_k$ be the transition labels along $[O',X']$.
By admissibility, $\sigma+\nu'$ and $[O',X']+\nu$ lie in $Y$ and share the vertex
$O+\nu'=O'+\nu$.

The sequence $e'_1,\ldots,e'_k$ cannot agree term by term with the first $k$ transition labels along either arm of $\sigma$ from $O$.
Otherwise, $X'+\nu'\in S$ would be adjacent to both $X'$ and the corresponding vertex
$V\in\sigma$, although
\[
d_S(X',V)\ge d_S(O,O')-2k\ge3,
\]
contradicting inducedness.

If $e'_1$ differs from the first labels on both arms of $\sigma$, the common translated vertex has three distinct neighbors in $Y$.
Otherwise, follow the matching arm to the first differing label, which exists by the preceding argument.
The new translated vertex is distinct from both neighbors of the preceding vertex along $\sigma+\nu'$: 
its label differs from the forward one by choice, and from the backward one because the subpath cannot immediately backtrack.
Thus four distinct vertices of $Y$ form a three-leaf star, with three in $\sigma+\nu'$
and one in $[O',X']+\nu$, as required in
\eqref{eq:outputroles}.
The complement-degree inequality follows from \cite[Lemma~3.1]{Zemor}.

We now exploit the distinction that $u_1(\pi), u_2(\pi), u_3(\pi)$ and $u_4(\pi)$ use $\pi$ in different roles.
Recall from Definition~\ref{def:heavy} that if a heavy segment has associated vertex $x$, center $O$, and label $\nu$, then
\[
\nu=x+O;
\]
thus any two of $x$, $O$, and $\nu$ determine the third.

Lemmas~\ref{lem:rolecongestion} and~\ref{lem:rolecollision} derive the combinatorial inequality \eqref{eq:parametriccut} by retaining two pieces of structure from the
preceding argument: the number of heavy complement vertices (rather than replacing
that count by aggregate degree information), and the distinction between the first three outputs of an admissible pair and its fourth output.
Keeping the corresponding multiplicity bounds separate permits the two roles to be weighted differently.

\begin{lemma}
\label{lem:rolecongestion}
For each vertex $u\in Y$, let $\ell_3(u)$ denote the total number of times $u$ occurs among
\[
u_1(\pi),\ u_2(\pi),\ u_3(\pi)
\]
as $\pi$ ranges over all admissible ordered pairs, and let $\ell_4(u)$ denote the number of times $u=u_4(\pi)$.
Then
\begin{equation}
\label{eq:rolecongestion}
\ell_3(u)\le s(u),
\qquad
\ell_4(u)\le s(u).
\end{equation}
\end{lemma}

\begin{proof}
Fix an occurrence of $u$, and let $e$ be the label of the edge from $u$ to the vertex of $S$ specified in \eqref{eq:outputroles}.
Thus $e=\nu'$ for a first-three occurrence and $e=\nu$ for a fourth occurrence.

First suppose that $u$ occurs among $u_1,u_2,u_3$ for an admissible pair $\pi=(\sigma,\sigma')$.
Then $u+e\in\sigma$.
Since the heavy segments in $\mathcal H$ are pairwise vertex-disjoint, the vertex
$u+e$ determines $\sigma$ uniquely, and hence also its associated vertex $x$.
By admissibility, the other segment $\sigma'$ has the same associated vertex $x$ and has label $e$.
By the definition of a heavy segment, its center $O'$ therefore satisfies
\[
x+O'=e.
\]
Since addition is modulo $2$, it follows that
\[
O'=x+e.
\]
Hence $\sigma'$ is also uniquely determined, since two members of the pairwise vertex-disjoint family $\mathcal H$
cannot have the same center.
Thus $u$ and the edge $u(u+e)$ determine the ordered pair $\pi$ uniquely in the first-three role.

The fourth role is analogous.
If $u=u_4(\pi)$, then $u+e\in[O',X']\subseteq\sigma'$, so $u+e$ determines $\sigma'$ and
its associated vertex $x$ uniquely.
Now $\sigma$ has label $e$, hence its center is $O=x+e$, which determines $\sigma$ uniquely.
Thus the same edge determines the ordered pair uniquely in the fourth role as well.

Finally, the four outputs associated with a fixed ordered pair are distinct.
Therefore a given edge from $u$ to $S$ accounts for at most one first-three occurrence of $u$ and at most one fourth
occurrence.
Since $u$ has exactly $s(u)$ such edges,
$\ell_3(u)\le s(u)$, and $\ell_4(u)\le s(u)$.
\end{proof}
The two inequalities in \eqref{eq:rolecongestion} refine the total multiplicity estimate in \cite{Zemor}: keeping the two roles separate allows them to be weighted differently.

For a real number $r$, write $(r)_+=\max\{r,0\}$.
We will assign weight $4\delta(u)/n$ to a first-three occurrence of $u$
and weight $(4\delta(u)/n-1)_+$ to a fourth occurrence.
The coefficients below collect these two contributions for a vertex with
$s(u)=m$; the proof of the next lemma explains the choice of weights and
derives the resulting inequality.
For $0\le m\le n-1$, define
\begin{equation}
\label{eq:pm}
p_m=
\begin{cases}
7m-\dfrac{8m^2}{n},&0\le m\le M,\\[2mm]
\dfrac{4m(n-m)}{n},&M<m\le n-1,
\end{cases}
\end{equation}
\begin{lemma}
\label{lem:rolecollision}
For every $a>8$, the count variables of a maximum coil satisfy
\begin{equation}
\label{eq:parametriccut}
\sum_{m=0}^{n-1}p_m y_m+3a^2\eta(y)
\ge \frac{a-8}{3}(S_n-10).
\end{equation}
\end{lemma}

\begin{proof}
For $u\in Y$, put $t_u=\delta(u)/n$.
Now, for an admissible pair $\pi$, write
\[
t_i=t_{u_i(\pi)},\qquad i=1,\ldots,4.
\]
By \eqref{eq:outputroles},
\[
t_1+t_2+t_3+t_4\ge1.
\]
Give an occurrence in one of the first three roles weight $4t_u$, and an occurrence in the fourth role weight $(4t_u-1)_+$.
The total weight assigned to $\pi$ is therefore
\[
4(t_1+t_2+t_3)+(4t_4-1)_+.
\]
If $t_4\ge1/4$, this equals
\[
4(t_1+t_2+t_3+t_4)-1\ge3.
\]
If $t_4<1/4$, then
\[
t_1+t_2+t_3\ge1-t_4>\frac34,
\]
so the first three roles alone have total weight greater than $3$.
Thus every admissible pair receives total weight $\ge 3$.
Note that the threshold $t_4=1/4$ is exactly the heavy-vertex threshold: when the fourth output is heavy it can absorb the subtraction of one, otherwise the first three outputs must carry more than three quarters of the total normalized complement degree.
Using \eqref{eq:rolecongestion} and grouping vertices according to $s(u)=m$ therefore gives
\[
3J\le\sum_{m=0}^{n-1}p_my_m,
\]
since
\[
m\left[4\frac{n-m}{n}+\left(4\frac{n-m}{n}-1\right)_+\right]=p_m.
\]

From \eqref{eq:paircount} and \eqref{eq:multiplicities},
\[
J\ge\frac{H^2}{K}-16H.
\]
For fixed $H$ and $K$, the function $2aH-a^2K$
is a concave quadratic in $a$, with maximum $H^2/K$.
Consequently,
\[
\frac{H^2}{K}\ge2aH-a^2K
\]
for every $a>0$.  
Hence, for $a>8$,
\[
\sum_{m=0}^{n-1}p_my_m
\ge6(a-8)H-3a^2K.
\]
Every associated vertex has complement degree at least $n/4$, so $K\le\eta(y)$.  Using this and \eqref{eq:packing}, whose coefficient is positive because $a>8$, yields
\[
\sum_{m=0}^{n-1}p_my_m+3a^2\eta(y)
\ge6(a-8)H
\ge\frac{a-8}{3}(S_n-10),
\]
which is the desired inequality \eqref{eq:parametriccut}.
\end{proof}

\section{A linear programming upper bound for \texorpdfstring{$S_n$}{Sn}}\label{sec:LP}

\subsection{The linear program}\label{subsec:LPmodel}
Fix $a>8$ and put
\[
q=n^2-3n+3,
\qquad
\beta_a=\frac{a-8}{3}.
\]
Using $p_m$ from \eqref{eq:pm}, consider the following linear program
with variables $z,y_0,\ldots,y_{n-1}$:
\[
\begin{array}{r@{\quad}rcl@{\quad}l}
(\mathrm P_n(a))
&\text{maximize}&z&&\\[1mm]
&\text{subject to}
&\displaystyle\sum_{m=0}^{n-1}y_m
&=2N-z&\text{\rm(C1)}\\[2mm]
&&\displaystyle\sum_{m=0}^{n-1}my_m
&=(n-2)z&\text{\rm(C2)}\\[2mm]
&&(n-1)y_{n-1}+4y_{n-2}-y_{n-3}
&\le2z&\text{\rm(C3)}\\[1mm]
&&qz-\displaystyle\sum_{m=0}^{n-1}\binom m2y_m
&\le\displaystyle\binom n2N&\text{\rm(C4)}\\[2mm]
&&\displaystyle\sum_{m=0}^{M}my_m
&\ge z/8&\text{\rm(C5)}\\[2mm]
&&\displaystyle\sum_{m=0}^{n-1}p_my_m
       +3a^2\sum_{m=0}^{M}y_m
&\ge\beta_a(z-10)&\text{\rm(C6$_a$)}\\[2mm]
&&y_m&\ge0\quad(0\le m\le n-1)&\text{\rm(C7)}.
\end{array}
\]
Observe that while $z$ is formally unrestricted in sign, (C2) and (C7) imply $z\ge0$.

A maximum coil determines a feasible point by taking $z=S_n$ and
$y_m=|Y_m|$. The identity \eqref{eq:Yn} makes (C1)--(C2) precisely
\eqref{eq:Ycount}--\eqref{eq:SYedgecount}. Constraints (C3)--(C5) are
Lemmas~\ref{lem:Yn2refinement}--\ref{lem:Zprojection}, and (C6$_a$) is
\eqref{eq:parametriccut}; nonnegativity gives (C7). Thus
\begin{equation}
\label{eq:parametricfeasible}
S_n\le\OPT(\mathrm P_n(a))
\qquad(a>8).
\end{equation}
We will determine this optimal value on an explicit range of parameters.

\subsection{A one-parameter family of primal solutions}
\label{subsec:primalfamily}
Before presenting an optimal family of solutions to $(\mathrm P_n(a))$ (and hence deriving upper bounds on $S_n$ from their objective values) we will explain how they were discovered.
Numerical experiments in small dimensions suggested an optimal solution supported on $\{0,M,n-3,n-2\}$, with (C3)--(C5) tight.
We also observed that (C5) is tight, but it seems plausible that as $n$ grows constraint (C6$_a$) may become binding for some or all values of $a$.
This is motivated combinatorially by the fact that (C5) and (C6$_a$) are similar, both having been derived from the machinery introduced in \cite{Zemor}, which focused on an asymptotic upper bound of $S_n$.

So we assume the following ansatz:
the primal LP has a solution with only $z,y_0,y_M,y_{n-3},y_{n-2}$ nonzero;
constraints (C1)--(C4) are always tight;
as $n$ grows the constraint (C6$_a$) may displace (C5) as the 5th active constraint.

To simplify notation and make the underlying ideas more clear, we introduce the following helper notation:
\begin{equation}
\label{eq:LPconstants}
\begin{aligned}
d&=5n-14,& b&=n(n-1)d,\\
T&=5n^3-17n^2+4n+12,&
E&=5n^2-28n+40-dM,\\
R&=n^2-14n+32.
\end{aligned}
\end{equation}
\noindent For technical reasons, observe that for $n\ge 13$ we have
\[
E\ge\frac54R>0,
\qquad
T-b=2(n-2)(n-3)>0.
\]

The ansatz leaves us with 5 variables, 4 fixed equations (C1)--(C4) and a varying 5th equation (C5) or (C6$_a$).
Rather than parameterizing directly on $n$ and $a$, let us introduce the parameter:
\[
\xi=\frac{My_M}{z}.
\]
\noindent Then, solving for (C1)--(C4) in the ansatz gives
\begin{equation}
\label{eq:primalfamily}
\begin{aligned}
z(\xi)&=\frac{Nb}{T+E\xi},
& y_M(\xi)&=\frac{\xi}{M}z(\xi),\\
y_{n-3}(\xi)&=\frac{2n-4-4\xi}{d}z(\xi),
& y_{n-2}(\xi)&=\frac{3n-8-\xi}{d}z(\xi),\\
y_0(\xi)&=2N-z(\xi)-y_M(\xi)-y_{n-3}(\xi)-y_{n-2}(\xi).
\end{aligned}
\end{equation}
\noindent and all other $y_m = 0$.
It can be verified (see Lemma~\ref{lem:primalfamily}) that $z$ and $y_m \ge 0$ precisely when 
\begin{equation}
\label{eq:xirange}
0\le\xi\le\frac{n-2}{2}.
\end{equation}
\noindent Meanwhile, the slacks corresponding to (C5) and (C6$_a$) are:
\begin{equation*}
\begin{aligned}
s_5(\xi)&=z(\xi)(\xi-1/8),\\
s_{6a}(\xi)&=\sum_{m=0}^{n-1}p_my_m(\xi)+3a^2\eta(y(\xi))
              -\beta_a\bigl(z(\xi)-10\bigr).
\end{aligned}
\end{equation*}
\noindent A feasible solution requires $s_5(\xi)\ge0$ and $s_{6a}(\xi)\ge0$.
For $s_5(\xi)$ this is obviously $\xi\ge 1/8$.
For $s_{6a}(\xi)$, substitution of \eqref{eq:primalfamily} gives:
\begin{equation}
\label{eq:xislack}
s_{6a}(\xi)=z(\xi)\bigl(\sigma_0(a)+\sigma_1(a)\xi\bigr),
\end{equation}
where
\begin{equation}
\label{eq:sigmadef}
\begin{aligned}
\sigma_0(a)
&=\frac{8(n-2)(6n-17)}{nd}-\beta_a
  +\frac{6a^2(n^2-9n+12)}{b}
  +\frac{10\beta_a T}{Nb},\\
\sigma_1(a)
&=6-\frac{8M}{n}+\frac{5R}{nd}
  +3a^2\left(\frac5d+\frac{2E}{b}\right)
  +\frac{10\beta_a E}{Nb}.
\end{aligned}
\end{equation}
\noindent In particular, $\sigma_1(a)>0$ since $M\le 3n/4$ and $E,R > 0$.
Thus $s_{6a}(\xi)\ge 0$ requires
\[
\xi\ge\xi_6(a):=-\frac{\sigma_0(a)}{\sigma_1(a)}.
\]
\noindent Finally, observe that since $E>0$, $z(\xi)$ is a decreasing function of $\xi$.
Hence we may assume $\xi$ is chosen as the smallest value $\xi^*(a)$ for which $s_5(\xi)\ge0$ and $s_{6a}(\xi)\ge0$.
Thus 
\begin{equation}
\label{eq:xistar}
\xi^*(a)=\max\left\{1/8,\xi_6(a)\right\}.
\end{equation}
\noindent Thus, for each fixed $a>8$, the value of $\xi^*(a)$ is determined entirely by $n$.

The choice of $\xi^*(a)$ determines $z(\xi), y_0(\xi), y_M(\xi), y_{n-3}(\xi), y_{n-2}(\xi)$ and hence the objective function value of this primal feasible solution.

\begin{lemma}
\label{lem:primalfamily}
For every $a>8$, the point $z, y_0, y_M, y_{n-3},  y_{n-2}$ from \eqref{eq:primalfamily} with
$\xi=\xi^*(a)$ is feasible for $(\mathrm P_n(a))$ and has objective value
\[
N\min\{L_n,w_n(a)\}
\]
where
\begin{align}
L_n
&:=\frac{b}{T+E/8}
=\frac{8n(n-1)(5n-14)}
 {40n^3-131n^2+4n+136-(5n-14)M},
\label{eq:Ln}\\
w_n(a)
&:=\frac{b\sigma_1(a)}
 {T\sigma_1(a)-E\sigma_0(a)}.
\label{eq:wna}
\end{align}
\noindent In addition, the four coordinates $y_0, y_M, y_{n-3}, y_{n-2}$ are positive.
\end{lemma}
\begin{proof}
First we verify that (C1)--(C4) are satisfied and $(z, y_0, y_M, y_{n-3},  y_{n-2})\ge 0$ when $\xi^*(a)\in \left[0, \frac{n-2}{2}\right]$.
Direct substitution into (C1)--(C4) verifies these four equalities.
Verifying non-negativity of $y_M, y_{n-3},  y_{n-2}$ is trivial, the only nontrivial nonnegativity check is $y_0$.
From \eqref{eq:primalfamily},
\begin{equation}
\label{eq:y0xi}
\frac{y_0}{z}
=\frac{2(n^2-9n+12)}{b}
+\left(\frac{2E}{b}+\frac5d-\frac1M\right)\xi.
\end{equation}
For fixed $n$, the coefficient of $\xi$ is a concave function of $M$.
At the endpoints of $3(n-1)/4\le M\le3n/4$, its values are,
respectively,
\[
\frac{5n^2-83n+354}{6b},
\qquad
\frac{5n^2-88n+368}{6b}.
\]
The first exceeds the second by $d/(6b)$, and the second is positive
because its numerator is
\[
5(n-13)^2+42(n-13)+69.
\]
Since $n^2-9n+12>0$, \eqref{eq:y0xi} gives $y_0>0$ for $\xi\ge0$.
The coordinates $y_M$ and $y_{n-3}$ impose
$0\le\xi\le(n-2)/2$; throughout this interval,
$y_{n-2}\ge z/2>0$. This proves \eqref{eq:xirange}.
The inequality $T>b$ also gives $0<z<N$.

To verify \eqref{eq:xislack}--\eqref{eq:sigmadef}, observe that
\[
\frac{\eta(y)}{z}
=\frac{2(n^2-9n+12)}{b}
+\left(\frac5d+\frac{2E}{b}\right)\xi,
\qquad
\frac{10\beta_a}{z}
=\frac{10\beta_a(T+E\xi)}{Nb}.
\]
Using \eqref{eq:pm} gives
\[
\frac{\sum_m p_my_m}{z}=c_0+c_1\xi,
\]
where
\[
c_0=\frac{8(n-2)(6n-17)}{nd},
\qquad
c_1=7-\frac{8M}{n}-\frac{8(7n-20)}{nd}
=6-\frac{8M}{n}+\frac{5R}{nd}>0.
\]
Combining these identities gives the claimed slack formula.
It also yields
\[
\begin{aligned}
T\sigma_1(a)-E\sigma_0(a)
={}&T\left(c_1+\frac{15a^2}{d}\right)\\
&+E\left[\beta_a+6a^2\left(1+\frac1d\right)-c_0\right]>0.
\end{aligned}
\]
Here $c_0<10$, since $10nd-8(n-2)(6n-17)=2n^2+92n-272>0$,
whereas $6a^2>384$.
Thus \eqref{eq:wna} is positive and well-defined for every $a>8$.

Next we check that $\xi^*(a)\in \left[0, \frac{n-2}{2}\right]$.
At $\xi=(n-2)/2$, \eqref{eq:primalfamily} gives
\[
y_{n-3}=0,\qquad y_{n-2}=z/2,\qquad
\eta(y)=2N-\frac32z>\frac12z,
\]
where the last inequality follows from $z<N$.
Since $p_m\ge0$ and $a>8$,
\[
s_{6a}((n-2)/2)>
\left(\frac32a^2-\beta_a\right)z+10\beta_a>0.
\]
By \eqref{eq:xislack}, $\xi_6(a)<(n-2)/2$, and hence
$0<\xi^*(a)<(n-2)/2$.

The point \eqref{eq:primalfamily} therefore satisfies (C1)--(C4) and (C7), and $y_0, y_M, y_{n-3}, y_{n-2}>0$.
Equations \eqref{eq:xislack}--\eqref{eq:xistar} ensure (C5)--(C6$_a$).
Finally, $z(1/8)=NL_n$ and $z(\xi_6(a))=Nw_n(a)$.
Since $\sigma_1(a)>0$ and
$T\sigma_1(a)-E\sigma_0(a)>0$, as established above,
\[
T+E\xi_6(a)
=\frac{T\sigma_1(a)-E\sigma_0(a)}{\sigma_1(a)}>0.
\]
Thus both evaluations lie on the same decreasing branch of $z(\xi)$. 

Recalling that
$\xi^*(a)=\max\{1/8,\xi_6(a)\}$, 
we obtain
\[
z(\xi^*(a))
=\min\{z(1/8),z(\xi_6(a))\}
=N\min\{L_n,w_n(a)\},
\]
as claimed.
The second identity is an algebraic evaluation; the point at $\xi_6(a)$ is used only when $\xi_6(a)\ge1/8$.
\end{proof}

If $\xi_6(a)<1/8$, then (C5) is tight and (C6$_a$) has positive slack. 
If $\xi_6(a)>1/8$, the reverse holds. 
At equality, both rows are tight. 
It remains to show that the constructed solution is optimal over the
whole feasible region, and thus provides an upper bound on $S_n$.

\subsection{Dual solutions and optimality}\label{subsec:dual}
Put
\[
\alpha_m=(n-1)\ind_{\{m=n-1\}}
+4\ind_{\{m=n-2\}}-\ind_{\{m=n-3\}},
\qquad
r_m(a)=p_m+3a^2\ind_{\{m\le M\}}.
\]
Assign dual variables $\lambda_1,\ldots,\lambda_6$ to
(C1)--(C6$_a$), respectively.
The equality-row variables $\lambda_1,\lambda_2$ are free, while
$\lambda_3,\lambda_4,\lambda_5,\lambda_6\ge0$.
The dual is
\[
\begin{aligned}
(\mathrm D_n(a))\qquad\text{minimize}\quad
&2N\lambda_1+\binom n2N\lambda_4+10\beta_a\lambda_6\\[1mm]
\text{subject to}\quad
&\lambda_1-(n-2)\lambda_2-2\lambda_3+q\lambda_4
 +\frac{\lambda_5}{8}+\beta_a\lambda_6=1,\\
&\lambda_1+m\lambda_2+\alpha_m\lambda_3
 -\binom m2\lambda_4-m\ind_{\{m\le M\}}\lambda_5
 -r_m(a)\lambda_6\ge0,\\
&\hspace{7cm}0\le m\le n-1,\\
&\lambda_3,\lambda_4,\lambda_5,\lambda_6\ge0,
 \qquad\lambda_1,\lambda_2\ \text{free}.
\end{aligned}
\]
The first dual constraint is an equality because $z$ is free.

The positive coordinates in Lemma~\ref{lem:primalfamily} suggest
imposing equality in the dual constraints at $0,M,n-3,n-2$.
If (C5) is tight and (C6$_a$) is slack, complementary slackness
requires $\lambda_6=0$; in the other regime it requires $\lambda_5=0$.
Solving the four tight equations with each of these choices gives the
two columns below, before the common scale is fixed.
Write
\[
\gamma=3a^2n-R.
\]
\begin{equation}
\label{eq:dualcolumns}
\begin{array}{c|cc}
\text{coordinate}&\ell&\rho\\ \hline
\lambda_1&0&(R+\gamma)E\\[1mm]
\lambda_2&(n-3)(5n-18)&(7n-8)E+5(M-1)\gamma\\[1mm]
\lambda_3&(n-3)(n-2)&(n-4)E+(2n-5-M)\gamma\\[1mm]
\lambda_4&2d&2(8E+5\gamma)\\[1mm]
\lambda_5&E&0\\[1mm]
\lambda_6&0&nE
\end{array}
\end{equation}
The normalizing denominators imposed by the dual equality are
\[
D_L=T+\frac E8,
\qquad
\mathcal D_n(a)=
\rho_1-(n-2)\rho_2-2\rho_3+q\rho_4+\beta_a\rho_6.
\]
Here the subscripts on $\rho$ refer to the coordinates in
\eqref{eq:dualcolumns}.
We use the parameter set
\begin{equation}
\label{eq:admissible_set}
\mathcal A_n=\{a>8:3a^2n\ge R\}.
\end{equation}
The inequality (C6$_a$) is valid for every $a>8$;
\eqref{eq:admissible_set} is the range on which the second dual
solution will be feasible.

\begin{lemma}\label{lem:dual_solutions}
For every $a>8$, the vector $\ell/D_L$ is feasible for
$(\mathrm D_n(a))$ and has objective value $NL_n$.
For $a\in\mathcal A_n$, the vector $\rho/\mathcal D_n(a)$ is also
feasible and has objective value $Nw_n(a)$.
Both vectors have zero slack in the dual constraints indexed by
$0,M,n-3,n-2$.
\end{lemma}
\begin{proof}
For the first column of \eqref{eq:dualcolumns}, direct substitution
into the dual equality gives $D_L=T+E/8>0$.
The multipliers $\lambda_3,\lambda_4,\lambda_5$ are positive, and
$\lambda_6=0$.
Let $\phi_\ell(m)$ denote the slack at $y_m$ after division by $D_L$.
Then
\[
D_L\phi_\ell(m)=
\begin{cases}
dm(M-m),&0\le m\le M,\\
m\bigl((n-3)(5n-18)-d(m-1)\bigr),&M<m\le n-4,\\
0,&m=n-3,n-2,\\
(n-1)R,&m=n-1.
\end{cases}
\]
For the second range, the expression in parentheses is at least
\[
(n-3)(5n-18)-d(n-5)=2(3n-8)>0.
\]
Thus every dual inequality is satisfied, with equality precisely at
$0,M,n-3,n-2$.
The objective value is
\[
2N\lambda_1+\binom n2N\lambda_4+10\beta_a\lambda_6
=\frac{Nb}{D_L}=NL_n.
\]

Now suppose $a\in\mathcal A_n$, so $\gamma\ge0$.
For the second column, simplification gives
\[
\begin{aligned}
\mathcal D_n(a)
={}&E\bigl(9n^2-28n+40+3a^2n+n\beta_a\bigr)\\
&+\gamma\bigl(10n^2-29n+30-M(5n-12)\bigr)>0.
\end{aligned}
\]
The first bracket is positive, and the second is at least
$25n^2/4-20n+30>0$.
The formulas in \eqref{eq:dualcolumns} therefore give
$\lambda_3,\lambda_4,\lambda_6>0$ and $\lambda_5=0$.

Let $\phi_\rho(m)$ denote the slack at $y_m$ after division by
$\mathcal D_n(a)$.
For $0\le m\le M$, direct substitution gives
\[
\mathcal D_n(a)\phi_\rho(m)
=5\gamma m(M-m)\ge0,
\]
where the inequality follows from $\gamma\ge0$.
The restriction \eqref{eq:admissible_set} follows naturally
from this calculation.

To check the remaining inequalities, put
\[
\varepsilon=3n-4M\in\{0,1,2,3\},
\qquad
E=\frac{5R+d\varepsilon}{4}.
\]
For $M<m\le n-4$, write $j=n-m$, so $4\le j<n/4$, and define
\[
f(j)=-4j^2+5nj-14n+32.
\]
The concave quadratic $f$ is positive on $[4,n/4]$, since
\[
f(4)=6n-32>0,\qquad f(n/4)=R>0.
\]
The dual slack has the factorization
\[
\mathcal D_n(a)\phi_\rho(m)
=Ef(j)+\frac{\gamma}{4}
  \bigl(5f(j)+\varepsilon(5j-14)\bigr)>0.
\]
The constraints at $n-3$ and $n-2$ are tight by construction, while
\[
\mathcal D_n(a)\phi_\rho(n-1)
=ER+\frac{\gamma}{4}
  \bigl(5R+\varepsilon(n-10)\bigr)>0.
\]
This proves feasibility. At $\gamma=0$, all the constraints indexed
by $0\le m\le M$ are tight.

Finally, substitution from \eqref{eq:sigmadef} and \eqref{eq:dualcolumns} gives
\[
\mathcal D_n(a) =n\bigl(T\sigma_1(a)-E\sigma_0(a)\bigr),\qquad
2\rho_1+\binom n2\rho_4+\frac{10\beta_a}{N}\rho_6 =nb\sigma_1(a).
\]
Their quotient is $w_n(a)$, so the dual objective is $Nw_n(a)$.
\end{proof}

\begin{theorem}\label{thm:LPoptimal}
For every $n\ge13$ and $a\in\mathcal A_n$,
\begin{equation}
\label{eq:LPvalue}
S_n\le\OPT(\mathrm P_n(a))=
N\min\{L_n,w_n(a)\}.
\end{equation}
An optimal primal solution is given by \eqref{eq:primalfamily} with $\xi=\xi^*(a)$.
\end{theorem}

\begin{proof}
Lemma~\ref{lem:dual_solutions} and weak duality give
\[
\OPT(\mathrm P_n(a))\le N\min\{L_n,w_n(a)\}.
\]
Lemma~\ref{lem:primalfamily} supplies a feasible primal point with
this objective value, proving equality. The bound for $S_n$ follows
from \eqref{eq:parametricfeasible}.
\end{proof}

\subsection{Comparing \texorpdfstring{$L_n$ and $w_n(a)$}{Ln and wn(a)}}
\label{subsec:parameter}
We would like to identify the threshold dimension $n$ for which $w_n(a)$ becomes a tighter upper bound on $S_n$ than $L_n$ for some choice of parameter $a$.
To do so, we consider the infimum of $w_n(a)$ over all admissible values of $a$.
Define
\[
\widehat w_n=\inf_{a\in\mathcal A_n}w_n(a).
\]
The relation between the two expressions in \eqref{eq:LPvalue}
can be read directly from \eqref{eq:xislack}:
\begin{equation}
\label{eq:slackcomparison}
w_n(a)\ge L_n
\quad\Longleftrightarrow\quad
\sigma_0(a)+\frac{\sigma_1(a)}8\ge0
\quad\Longleftrightarrow\quad
s_{6a}(1/8)\ge0.
\end{equation}
Thus the comparison tests whether (C6$_a$) is satisfied at the
point where (C5) is tight.

\begin{corollary}\label{cor:LPbound}
For every $n\ge13$,
\begin{equation}
\label{eq:Sn_envelope}
S_n\le N\min\{L_n,\widehat w_n\}.
\end{equation}
Moreover,
\begin{equation}
\label{eq:crossover}
\widehat w_n\ge L_n\quad(13\le n\le780),
\qquad
\widehat w_n<L_n\quad(n\ge781).
\end{equation}
For $n\ge781$, the explicit admissible choice $a=2n/21$ satisfies
$w_n(a)<L_n$.
\end{corollary}

\begin{proof}
Taking the infimum over $a\in\mathcal A_n$ in
Theorem~\ref{thm:LPoptimal} proves \eqref{eq:Sn_envelope}.
Appendix~\ref{app:certificatecrossover} proves \eqref{eq:crossover}
by determining the sign in \eqref{eq:slackcomparison}.
\end{proof}

The bound \eqref{eq:Sn_envelope} is therefore $NL_n$ through dimension $780$ and is strictly smaller from dimension $781$ onward.

\section{Comparison with existing upper bounds}
\label{sec:comparison}

We now compare \eqref{eq:Sn_envelope} with the three upper bounds $U_L(n),U_{LvZ}(n),U_Z(n)$ defined in Section~\ref{sec:intro}.
We first consider $NL_n$, and then analyze $w_n(a)$ using a particular
admissible (but not necessarily optimal) choice of parameter $a$.

\subsection{The bound \texorpdfstring{$NL_n$}{NLn}}

\begin{lemma}\label{lem:lvzcomparison}
For every $n\ge13$,
\[
NL_n<\min\{U_L(n),U_{LvZ}(n)\}.
\]
In addition, $NL_n<U_Z(n)$ for $13\le n\le19853$.
\end{lemma}

\begin{proof}
Since $E>0$ and $T-b=2(n-2)(n-3)>0$, we have
\[
L_n=\frac{b}{T+E/8}<\frac bT<1.
\]
For the first two comparisons, it is enough to use the larger
quantity $b/T$.
For comparison with \eqref{eq:Lukito},
\[
(4n-10)T-(4n-9)b
=(n-13)(3n^2-2n+108)+1284>0.
\]
Thus
\[
L_n<\frac bT<\frac{4n-10}{4n-9}=\frac{U_L(n)}N.
\]
For comparison with \eqref{eq:LvZ}, the additive term is positive,
and
\[
5T-(5n-14)(5n^2-3n-20)=78n-220>0.
\]
Consequently,
\[
L_n<\frac bT
<\frac{5n(n-1)}{5n^2-3n-20}
<\frac{U_{LvZ}(n)}N.
\]

For comparison with $U_Z(n)$, first observe that $U_Z(n)>N$
when $n\le10584$, since $\sqrt{n/6}\le42$.
Thus $L_n<1$ suffices in this range.
For the remaining dimensions it is enough to prove
\[
L_n<\frac{6n}{6n+\sqrt{n/6}/6-7},
\]
or, equivalently,
\[
\sqrt{\frac n6}
<
42+36n\left(\frac1{L_n}-1\right).
\]
Using $M\le3n/4$ in \eqref{eq:Ln} gives
\[
\begin{aligned}
42+36n\left(\frac1{L_n}-1\right)
&\ge
\frac{2301}{40}
-\frac{22}{5(5n-14)}
-\frac{223}{8(n-1)}\\
&>
\frac{2301}{40}-\frac{29}{n}.
\end{aligned}
\]
Here the two subtracted terms are less than $1/n$ and $28/n$,
respectively, for $n\ge10585$.
The function
\[
h(x)=\frac{2301}{40}-\frac{29}{x}-\sqrt{\frac x6}
\]
is strictly decreasing for $x\ge10585$, since
\[
h'(x)=\frac{29}{x^2}-\frac1{2\sqrt{6x}}<0.
\]
At the upper endpoint,
\[
\frac{2301}{40}-\frac{29}{19853}
>
\frac{57523}{1000}
>
\sqrt{\frac{19853}{6}},
\]
where the last inequality follows from
\[
6(57523)^2-19853(1000)^2=373174>0.
\]
Hence $h(n)>0$ for $10585\le n\le19853$, completing the proof.
\end{proof}

Since $M=3n/4+O(1)$, the definitions in \eqref{eq:LPconstants}
give $E=5n^2/4+O(n)$ and $T=5n^3+O(n^2)$. Thus
\[
1-L_n
=\frac{2(n-2)(n-3)+E/8}{T+E/8}
=\frac{69}{160n}+O(n^{-2}).
\]
By comparison, \eqref{eq:LvZ} gives
\[
\frac{U_{LvZ}(n)}N=1-\frac{2}{5n}+O(n^{-2}).
\]
The coefficient of the $1/n$ correction therefore improves from
$2/5=64/160$ to $69/160$.

\subsection{The bound \texorpdfstring{$Nw_n(a)$}{Nwn(a)}}

To compare $N\widehat w_n$ with $U_Z(n)$, it suffices to select
one admissible parameter.
At the boundary $3a^2n=R$ of \eqref{eq:admissible_set}, we have
$\gamma=0$, and every coordinate of $\rho$ in
\eqref{eq:dualcolumns} has the common factor $E$.
Its cancellation removes the dependence on $M$ from the objective.
For $n\ge206$, put
\begin{equation}
\label{eq:boundaryparameter}
a_n=\sqrt{\frac{R}{3n}}.
\end{equation}
Since
\[
a_n^2-64=\frac{n(n-206)+32}{3n}>0,
\]
we have $a_n\in\mathcal A_n$.
Substitution into \eqref{eq:dualcolumns} gives
\begin{equation}
\label{eq:boundarywn}
w_n(a_n)
=\frac{30n^2-108n+192+10n(a_n-8)/N}
       {30n^2-134n+216+na_n}.
\end{equation}

\begin{theorem}\label{thm:asymptotic}
For every $n\ge206$,
\begin{equation}
\label{eq:boundarycomparison}
Nw_n(a_n)<U_Z(n).
\end{equation}
Moreover, as $n\to\infty$,
\begin{equation}
\label{eq:Sn_asymp}
S_n\le
N\left(1-\frac1{30\sqrt{3n}}+O(n^{-1})\right).
\end{equation}
\end{theorem}

\begin{proof}
Write
\[
B=30n^2-108n+192,\qquad
D=30n^2-134n+216+na_n,\qquad
h=\frac{\sqrt n}{6\sqrt6}.
\]
All denominators below are positive.
The term involving $1/N$ in \eqref{eq:boundarywn} is smaller
than the corresponding term in $U_Z(n)/N$.
Indeed, $a_n<n$, and hence
\[
D-10n(a_n-8)
=30n^2-54n+216-9na_n
>21n^2-54n+216>0.
\]
Therefore
\[
w_n(a_n)<\frac BD+\frac1N,
\qquad
\frac{U_Z(n)}N=\frac{6n}{6n+h-7}+\frac1N.
\]
It remains to compare the two fractions.
Cross-multiplication gives
\[
6nD-B(6n+h-7)
=6n^2a_n+54n^2-612n+1344-Bh.
\]
Since
\[
a_n^2-\frac n4=\frac{n(n-56)+128}{12n}>0,
\]
we have $a_n>\sqrt n/2$.
Also $B<30n^2$ and $6\sqrt6>10$, so
\[
Bh<3n^2\sqrt n<6n^2a_n.
\]
Finally,
\[
54n^2-612n+1344=54n(n-12)+36n+1344>0.
\]
Thus $B/D<6n/(6n+h-7)$, proving \eqref{eq:boundarycomparison}.

For the asymptotic statement, observe that
\[
\frac{a_n}{n}
=\frac1{\sqrt{3n}}
 \left(1-\frac{14}{n}+\frac{32}{n^2}\right)^{1/2}
=\frac1{\sqrt{3n}}+O(n^{-3/2}).
\]
Dividing the numerator and denominator of \eqref{eq:boundarywn}
by $30n^2$, and noting that the term involving $1/N$ is
exponentially small, gives
\[
w_n(a_n)
=\frac{1+O(n^{-1})}
 {1+1/(30\sqrt{3n})+O(n^{-1})}
=1-\frac1{30\sqrt{3n}}+O(n^{-1}).
\]
Since $S_n\le N\widehat w_n\le Nw_n(a_n)$ by
\eqref{eq:Sn_envelope}, this proves \eqref{eq:Sn_asymp}.
\end{proof}

The choice of $a_n$ in \eqref{eq:boundaryparameter} can also be motivated by the leading asymptotic behavior of $w_n(a)$.
Since the second dual solution requires $3a^2n\ge R=n^2-14n+32$, its smallest admissible parameter satisfies $a_n\sim\sqrt{n/3}$.
This suggests considering parameters of the form $a=t\sqrt n$, where $t>0$ is fixed.
The condition $3a^2n\ge R$ then becomes
\[
(3t^2-1)n^2+14n-32\ge0,
\]
so such a parameter is admissible for all sufficiently large $n$ precisely when $t\ge1/\sqrt3$.

Substitution into \eqref{eq:dualcolumns} gives
\begin{equation}
\label{eq:wasymptotic_t}
w_n(t\sqrt n)
=1-\frac{t}{6(9t^2+2)\sqrt n}+O(n^{-1}).
\end{equation}
To obtain the strongest leading correction in \eqref{eq:wasymptotic_t},
we seek to maximize $t/(6(9t^2+2))$ over $t\ge1/\sqrt3$.
But
\[
\frac{d}{dt}\frac{t}{6(9t^2+2)}
=\frac{2-9t^2}{6(9t^2+2)^2}<0
\qquad (t\ge1/\sqrt3).
\]
Thus the maximum occurs at $t=1/\sqrt3$, yielding the coefficient $1/(30\sqrt3)$ in \eqref{eq:Sn_asymp}.
The exact boundary choice $a=a_n$ has this same leading behavior and, as observed above, also simplifies the finite-dimensional expression for $w_n(a)$.

Expanding the finite expression in \eqref{eq:Zemor} gives
\[
\frac{U_Z(n)}N
=1-\frac1{36\sqrt{6n}}+O(n^{-1}).
\]
The coefficient of the leading correction in \eqref{eq:Sn_asymp} therefore improves by the factor
\[
\frac{1/(30\sqrt3)}{1/(36\sqrt6)}
=\frac{6\sqrt2}{5}.
\]

\begin{corollary}\label{cor:globalcomparison}
For every $n\ge13$,
\[
S_n\le N\min\{L_n,\widehat w_n\}
<\min\{U_L(n),U_{LvZ}(n),U_Z(n)\}.
\]
\end{corollary}

\begin{proof}
Corollary~\ref{cor:LPbound} gives the first inequality.
Lemma~\ref{lem:lvzcomparison} proves the comparisons with
$U_L(n)$ and $U_{LvZ}(n)$ for every $n\ge13$, and with
$U_Z(n)$ for $13\le n\le19853$.
For $n\ge206$, Theorem~\ref{thm:asymptotic} gives
$N\widehat w_n\le Nw_n(a_n)<U_Z(n)$.
These ranges cover every $n\ge13$.
\end{proof}

\section{Concluding remarks}
\label{sec:conclusions}
Our analysis begins by building upon the square-counting results of \cite{EL} and the heavy-segment machinery of \cite{Zemor} to derive new valid combinatorial inequalities for maximum coils.
We then relax these (and other) inequalities into a parametric LP whose optimal value provides an upper bound for $S_n$.
This is in the spirit of linear-programming bounds in coding theory \cite{Delsarte,DGS,MacWilliams,Roth}, although the variables here are degree classes rather than distance distributions.
It also differs from direct combinatorial optimization models such as the QUBO formulation in \cite{Fuidio}, whose variables describe vertices and their embeddings in $Q_n$.

We derive two analytic upper bounds, $S_n\le NL_n$ and $S_n\le N\widehat w_n$, and establish that their combination
$S_n\le N\min\{L_n,\widehat w_n\}$ improves upon the prior upper bounds of \cite{Lukito,LvZ,Zemor} for all $n\ge13$.
The second bound also improves the leading asymptotic correction in \eqref{eq:Zemor}:
\[
S_n\le 2^{n-1}\left(1-\frac1{30\sqrt{3n}}+O(n^{-1})\right).
\]

Optimality in the linear program does not imply realizability as the degree counts of a coil, as the following corollary illustrates.

\begin{corollary}
\label{cor:coil_realizability}
Every maximum coil whose degree counts satisfy $y_0>0$ satisfies (C4)
strictly. In particular, for every $n\ge13$ and $a\in\mathcal A_n$,
the optimal solution constructed in Theorem~\ref{thm:LPoptimal} is not
realizable as the length and degree counts of a coil in $Q_n$.
\end{corollary}

\begin{proof}
Let $S$ be a maximum coil in $Q_n$ having degree counts
$y_m=|Y_m|$ for $m=0,\ldots,n-1$, and suppose that $y_0>0$.
We first show that $F_5>0$.
Without loss of generality, assume that $Y_0$ contains $x=\varnothing$.
If $F_5=0$, then, since $\varnothing\in Y_0$, every vertex of weight
two in $Q_n$ must belong to $S$.
However, any weight-two vertex of $S$ must be adjacent to two
weight-three vertices of $S$ (since all vertices of weight one lie
in $Y$), and each such weight-three vertex is adjacent to three
weight-two vertices of $S$, a contradiction.
By \eqref{eq:F5identity}, (C4) must therefore be strict.

The point constructed in Theorem~\ref{thm:LPoptimal} has $y_0>0$
and satisfies (C4) with equality. If realized by a coil, that
coil would necessarily be maximum because its length would equal
$\OPT(\mathrm P_n(a))\ge S_n$, contradicting the preceding paragraph.
\end{proof}
We observe that the preceding obstruction is combinatorial and does not rely on possible nonintegrality or parity of the LP solution.

For fixed $a\in\mathcal A_n$, let
$(\mathrm P_n^{\mathrm{int}}(a))$ denote the restriction of
$(\mathrm P_n(a))$ in which $y_0,\ldots,y_{n-1}$ are integral
and $z$ is an even integer.
One may ask how much of the gap between $\OPT(\mathrm P_n(a))$
and $S_n$ is removed by these arithmetic restrictions, and whether
an optimal solution of $(\mathrm P_n^{\mathrm{int}}(a))$ can be
realized as the degree counts of a coil in $Q_n$.
More generally, one may compare the convex hull of the
integer-feasible vectors with the convex hull of vectors realized by maximum coils.
Understanding the difference between these polytopes may distinguish
restrictions that are essentially arithmetic from those that
require the geometry of an induced cycle in the hypercube.

\appendix
\section{Proof of Lemma~\ref{lem:Yn2refinement}}
\label{app:Yn2refinement}
\begin{proof}
For convenience we identify each binary vector $v\in Q_n$ with the set ${\{i \in [n] \mid v_i = 1\}}$;
adjacency then corresponds to adding or deleting one element.
For $x\in Y$, let
\[
r(x)=|\{\text{$\mathcal{F}_1$-squares containing $x$}\}|.
\]
\noindent The proof is an incidence-counting argument.
The local configurations in Claims~\ref{cl:claim_2}--\ref{cl:claim_4} follow the case analysis in the proof of \cite[Lemma~4]{EL}; we include the details because the aggregate accounting below differs from the inequality derived there.

Since every $\mathcal{F}_1$-square contains a unique vertex in $Y$, \eqref{eq:F1} gives
\begin{equation}
\label{eq:rx_budget}
S_n=F_1=\sum_{x\in Y}r(x)\ge
\sum_{x\in Y_{n-1}}r(x)+\sum_{x\in Y_{n-2}}r(x).
\end{equation}

We will bound the two terms on the right separately.

\begin{claim}
\label{cl:claim_1}
We have
\[
\sum_{x\in Y_{n-1}} r(x)
\ge
\frac{n-1}{2}\,|Y_{n-1}|.
\]
\end{claim}

\begin{proof}
Fix $x\in Y_{n-1}$.  After translating the cube and permuting
coordinates, we may assume that
$x=\varnothing \text{ and } N(x)\cap Y=\{\{n\}\}$.
Thus $\{i\}\in S$ for every $i\in[n-1]$.

Count the two neighbors in $S$ of each of the $n-1$ vertices
$\{i\}$, $i\in[n-1]$.  Each vertex $\{i,j\}\in S$ with
$i,j\in[n-1]$ is counted twice and corresponds to an
$\mathcal F_1$-square containing $x$, while each vertex $\{i,n\}\in S$ is counted once.  
Recall from Section~\ref{sec:prelim} that $s(v)=|N(v)\cap S|$ denotes the number of neighbors of a vertex $v$ lying in $S$.
Hence
\[
2(n-1)=2r(x)+s(\{n\}).
\]
By \eqref{eq:Yn}, $Y_n=\varnothing$, and therefore
$s(\{n\})\le n-1$.  It follows that $r(x)\ge \frac{n-1}{2}$.
Summing over $x\in Y_{n-1}$ proves the claim.
\end{proof}

We now consider the contribution from $Y_{n-2}$.  
We first show that every $x\in Y_{n-2}$ satisfies $r(x)\ge 1$ (Claim~\ref{cl:claim_2}).  
We then analyze the exceptional case $r(x)=1$.  
Such a vertex is either balanced, in which case it
determines two vertices of $Y_{n-3}$ (Claim~\ref{cl:claim_3}), or
unbalanced, in which case it determines a distinct companion in
$Y_{n-2}$ lying in at least four $\mathcal F_1$-squares (Claim~\ref{cl:claim_4}).

Fix $x\in Y_{n-2}$ and, after translating the cube and permuting coordinates, we may assume
\[
x=\varnothing,\qquad
\{n-1\},\{n\}\in Y,\qquad
\{1\},\ldots,\{n-2\}\in S.
\]
Let $I=[n-2]$.  Since each singleton $\{i\}$, $i\in I$, lies on the
induced cycle and $x\notin S$, exactly two of the vertices
$\{i,j\}$, $j\ne i$, lie in $S$.  Moreover, a square through
$x=\varnothing$ is an $\mathcal{F}_1$-square precisely when it has the form
\[
\varnothing,\quad \{i\},\quad \{j\},\quad \{i,j\},
\qquad i,j\in I,
\]
with $\{i,j\}\in S$.  Hence
\begin{equation}
 r(x)=\left |\{\{i,j\}\subset I:\ \{i,j\}\in S\}\right |.
\label{eq:r-direct}
\end{equation}

\begin{claim}
\label{cl:claim_2}
Every $x\in Y_{n-2}$ satisfies $r(x)\ge1$.
\end{claim}
\begin{proof}
Suppose instead that $r(x)=0$.  
By \eqref{eq:r-direct}, no $\{i,j\}$ with $i,j\in I$ lies in $S$.  
Thus the two cycle-neighbors in $S$ of each $\{i\}$, $i\in I$, must be $\{i, n-1\} \text{ and } \{i,n\}$.
Also $\{i,n-1,n\}\in Y$, since otherwise the square on
\[
\{i\},\quad \{i,n-1\},\quad \{i,n\},\quad \{i,n-1,n\}
\]
would lie entirely in $S$.

For each omitted coordinate $a\in\{n-1,n\}$, the vertex
$\{i,a\}\in S$ has $\{i\}$ as one cycle-neighbor. Its other
cycle-neighbor must be $\{i,j,a\}$ for some $j\in I\setminus\{i\}$.
Since $\{j,a\}\in S$ as well, these two pair-vertices are the two
cycle-neighbors of the triple. Thus the triples pair the indices in
$I$ for each omitted coordinate. If $|I|$ is odd, these pairings
already give a contradiction. Otherwise there are $|I|/2$ triples
for each of the two omitted coordinates.

The singletons, the $2|I|$ pair-vertices, and these $|I|$ triples form
a nonempty subgraph in which every vertex already has its two
cycle-neighbors. No edge of $S$ leaves this subgraph. Since $S$ is
connected, it is the whole coil, and
\[
 S_n=|I|+2|I|+|I|=4(n-2).
\]
However, \cite{AK} gives $S_n\ge(77/256)2^n>4(n-2)$ for $n\ge13$.
The last inequality holds at $n=13$, and its left side doubles at
successive dimensions while $4(n-2)$ grows linearly. This
contradiction proves the claim.
\end{proof}

Suppose $r(x)=1$.  Then there is a unique pair $\{i,j\}\subset I$
with $\{i,j\}\in S$.  For every $k\in I\setminus\{i,j\}$, both
$\{k,n-1\}$ and $\{k,n\}$ lie in $S$, whereas for each of
$k=i,j$ exactly one of $\{k,n-1\},\{k,n\}$ lies in $S$.  Thus exactly
two vertices among
$\{k,n-1\}, \{k,n\}, k\in I$
lie in $Y$.

Put $a=\{n-1\}, b=\{n\}, c=\{n-1,n\}$.
For every $k\in I\setminus\{i,j\}$ we have $\{k,n-1,n\}\in Y$, since the
other three vertices of the square
\[
\{k\},\quad \{k,n-1\},\quad \{k,n\},\quad \{k,n-1,n\}
\]
lie in $S$.  We next show that $c\in Y$.  If instead $c\in S$, then
$a,b\in Y$ and the preceding observation imply that the only possible
cycle-neighbors of $c$ are $\{i,n-1,n\}$ and
$\{j,n-1,n\}$; hence both must lie in $S$.  If $w_i,w_j$ are the
unique $S$-vertices among
$\{i,n-1\},\{i,n\}$ and $\{j,n-1\},\{j,n\}$,
respectively, then
\[
\{i\}-\{i,j\}-\{j\}-w_j-\{j,n-1,n\}-c
-\{i,n-1,n\}-w_i-\{i\}
\]
is an $8$-cycle in $S$.  
But this forces $S_n=8$, impossible for $n\ge13$, therefore $c\in Y$.  
Moreover, since $\{n-1\},\{n\}\in Y$ and $\{k, n-1,n\}\in Y$ for $k\in I\setminus\{i,j\}$, $c$ has at most 2 neighbors in $S$, thus $c\not\in Y_{n-2}$.

Write $\delta(w)=|N(w)\cap Y|$.
Both $a$ and $b$ have $x$ and $c$ as neighbors in $Y$.
In addition, exactly 2 vertices among $\{k,n-1\}, \{k,n\}, k\in I$ are in $Y$.
Thus
$\delta(a)+\delta(b)=6$.
Since each degree is at least $2$, either
\[
\delta(a)=\delta(b)=3,
\]
which we call the \emph{balanced} case, or, after interchanging $a,b$,
\[
\delta(a)=2,\qquad \delta(b)=4,
\]
which we call the \emph{unbalanced} case.

To summarize the preceding discussion, each $x \in Y_{n-2}$ with $r(x)=1$
determines two vertices $a$ and $b$ as
$\{a,b\}=N(x)\cap Y$.
We call $x$ a \emph{balanced vertex} if $\delta(a)=\delta(b)=3$ and an
\emph{unbalanced vertex} if $\delta(a)=2$ and $\delta(b)=4$, or if
$\delta(a)=4$ and $\delta(b)=2$.

\begin{claim}[Balanced case]
\label{cl:claim_3}
Let $x$ be a balanced vertex with neighbors $a,b\in Y$. Then neither
$a$ nor $b$ is a neighbor in $Y$ of any other balanced vertex $x'$.
\end{claim}

\begin{proof}
By the previous translation and permutation argument, assume
$x=\varnothing, a=\{n-1\},  b=\{n\}$.
We will show that if $a\in N(x')\cap Y$ for some other balanced vertex
$x'$, we reach a contradiction; the case for $b$ is symmetric.

Since $\delta(a)=3$ and $x=\varnothing$ and $c=\{n-1,n\}$ are elements
of $N(a)\cap Y$, $a$ must have a third neighbor in $Y$ of the form
$\{i,n-1\}$ or $\{j,n-1\}$, where $\{i,j\}$ is the unique pair of
indices in $I$ with $\{i,j\}\in S$. 
As $\delta(a)=3$ and $x$ and $c$ already account for two neighbors of $a$ in $Y$, precisely one of $\{i,n-1\}$ and $\{j,n-1\}$ lies in $Y$.
Thus, without loss of generality,
\[
p=\{i,n-1\}\in Y
\qquad\text{and}\qquad
q=\{j,n-1\}\in S.
\]
Because $c=\{n-1,n\}\notin Y_{n-2}$, this forces $x'=p$.

Also note that $\{i,j,n-1\}\in Y$, since the vertex $\{i,j\}\in S$ already has its two cycle-neighbors
$\{i\}$ and $\{j\}$.
Therefore
\[
N(x')\cap Y=\bigl\{a,\{i,j,n-1\}\bigr\}.
\]
Translating $Q_n$ by $x'$ so that
\[
x'\mapsto\varnothing, \qquad
a\mapsto\{i\}, \qquad
\{i,j,n-1\}\mapsto\{j\},
\]
and applying the previous observation that $c\notin S$ gives us
$\{i,j\}\notin S$ in the translate, i.e.
$\{j,n-1\}\notin S$ in $Q_n$. But this contradicts the fact that
$q\in S$.
\end{proof}

Recall that for a balanced vertex $x$, its neighbors $a,b\in Y_{n-3}$.
Thus, letting $\mathcal{B}$ denote the set of balanced vertices and
$B=|\mathcal{B}|$, we have
\begin{equation}
\label{eq:balancedcount}
2B\le |Y_{n-3}|.
\end{equation}

\begin{claim}[Unbalanced case]
\label{cl:claim_4}
Every unbalanced vertex $x$ has a distinct companion
$a\in Y_{n-2}$ satisfying
\begin{equation*}
r(a)\ge \left\lceil\frac{n-4}{2}\right\rceil\ge4.
\end{equation*}
\end{claim}

\begin{proof}
Let $x$ be an unbalanced vertex. After translating $Q_n$ and permuting
$[n]$, we may assume $x=\varnothing,  a=\{n-1\}, b=\{n\}$,
with $\{a,b\}=N(x)\cap Y$.
We also assume without loss of generality that
$\delta(a)=2$ and $\delta(b)=4$.
Thus $a\in Y_{n-2}$ and, recalling that $c=\{n-1,n\}\in Y$, $a$'s
neighbors in $Y$ are $x$ and $c$. Since we showed earlier that
$c\notin Y_{n-2}$, we conclude that two distinct unbalanced vertices
cannot share a common companion $a$ with $\delta(a)=2$.

Because
\[
N(a)\cap Y=\{x,c\},
\]
we have
\[
\{k,n-1\}\in S
\qquad \forall\, k\in I.
\]
For each $k\in I\setminus\{i,j\}$, the vertex
\[
\{k,n-1,n\}\in Y.
\]
Thus in $S$ the neighbors of $\{k,n-1\}$ are $\{k\}$ and
$\{k,\ell,n-1\}$, where $\ell\in I\setminus\{k\}$. Since
$\{\ell,n-1\}$ is also in $S$, the vertices
\[
a,\qquad
\{k,n-1\},\qquad
\{k,\ell,n-1\},\qquad
\{\ell,n-1\}
\]
form an $\mathcal{F}_1$ square through $a$. Such a
$\mathcal{F}_1$ square exists for each of the $n-4$ elements
$k\in I\setminus\{i,j\}$, and each square arises from at most two choices of 
$k\in I\setminus\{i,j\}$; hence
\[
r(a)\ge
\left\lceil\frac{n-4}{2}\right\rceil
\ge 4,
\]
completing the proof.
\end{proof}

We now complete the proof. Let $\mathcal{U}$ be the set of unbalanced
vertices, let $U=|\mathcal{U}|$, and let $\mathcal{C}$ denote the set of companions from Claim~\ref{cl:claim_4}. 
By Claim~\ref{cl:claim_4}, $|\mathcal{C}|=U$, and each $a\in\mathcal{C}$ has $r(a)\ge 4$. 
Also, $\mathcal{C}$ is disjoint from $\mathcal{B}\cup\mathcal{U}$, as $r(a)\ge 4$. 
Following from Claim~\ref{cl:claim_2}, every
$x\in Y_{n-2}\setminus(\mathcal{B}\cup\mathcal{U})$
has $r(x)\ge 2$. 
Therefore, taking
\[
X=|Y_{n-2}|,
\qquad
R=\sum_{x\in Y_{n-2}} r(x),
\]
and partitioning $Y_{n-2}$ into
$\mathcal{B},
\mathcal{U},
\mathcal{C},
Y_{n-2}\setminus
(\mathcal{B}\cup\mathcal{U}\cup\mathcal{C}),
$
we have
\begin{align*}
R
&\ge B+U+4U+2(X-B-U-U)\\
&=2X-B+U\\
&\ge 2X-B.
\end{align*}

Combining this with \eqref{eq:balancedcount} gives
\begin{equation}
\label{eq:yn2_bound}
\sum_{x\in Y_{n-2}}r(x)
\ge
2|Y_{n-2}|-\frac12|Y_{n-3}|.
\end{equation}

Combining Claim~\ref{cl:claim_1} and \eqref{eq:yn2_bound} with the earlier bound \eqref{eq:rx_budget} gives
\[
\begin{aligned}
S_n
&\ge
\sum_{x\in Y_{n-1}} r(x)
+
\sum_{x\in Y_{n-2}} r(x)\\
&\ge
\frac{n-1}{2}|Y_{n-1}|
+
2|Y_{n-2}|
-
\frac12|Y_{n-3}|.
\end{aligned}
\]
Multiplying by $2$ proves the lemma.
\end{proof}

\section{The comparison in Corollary~\ref{cor:LPbound}}
\label{app:certificatecrossover}
Put $G_n=8T+E$. By \eqref{eq:Ln}, $L_n=8b/G_n$.
At $\xi=1/8$, the slack in (C6$_a$) is
\[
s_{6a}(1/8)=\frac{N}{3G_n}\Psi_n(a),
\]
where
\begin{align*}
\Psi_n(a)
={}&9\bigl(31n^2-205n+272-(10n-28)M\bigr)a^2\\
&-8n(n-1)(5n-14)a\\
&+(n-1)\bigl(1577n^2-6926n+7008-(120n-336)M\bigr)\\
&+\frac{10G_n}{N}(a-8).
\end{align*}
Consequently, $w_n(a)<L_n$ if and only if $\Psi_n(a)<0$.

For $a>8$, the last term of $\Psi_n(a)$ is positive, while its
remaining terms decrease with $M$.  Thus $M\le3n/4$ gives
\[
\begin{aligned}
\Psi_n(a)>{}&
\frac92(47n^2-368n+544)a^2
-8n(n-1)(5n-14)a\\
&+(n-1)(1487n^2-6674n+7008).
\end{aligned}
\]
The coefficient of $a^2$ on the right is positive for $n\ge13$.
The discriminant of this quadratic is $2(n-1)P(n)$, where
\begin{align*}
P(n)={}&800n^5-634281n^4+7758798n^3\\
&-32355296n^2+55886400n-34311168.
\end{align*}
For $13\le n\le780$,
\[
P(n)<n^3\left(
800n^2-634281n+7758798+\frac{55886400}{n^2}
\right)<0.
\]
Indeed, the expression in parentheses is convex on $[13,780]$.
At the endpoints its values are $-3543295/169$ and $-43989034/169$,
respectively. Both are negative. The quadratic is therefore positive for every
real $a$, so no $a\in\mathcal A_n$ improves $L_n$ in this range.

For $n\ge781$, take $a=2n/21$.  This value exceeds $8$ and lies in $\mathcal A_n$, since
\[
3a^2n-R=n^2\left(\frac{4n}{147}-1\right)+14n-32>0.
\]
Using $M\ge3(n-1)/4$ gives
\[
-147N\Psi_n(2n/21)
\ge NP_1(n)-35(n-84)(160n^3-539n^2+73n+502),
\]
where
\[
P_1(n)=278n^4-218599n^3+1184993n^2-1960980n+993132.
\]
Writing $n=781+t$ gives
\[
\begin{aligned}
P_1(n)={}&278t^4+649873t^3+506422484t^2\\
&+131572453761t+15750882804>0.
\end{aligned}
\]
Here $P_1(n)\ge15750882804$, while
\[
35(n-84)(160n^3-539n^2+73n+502)<5600n^4.
\]
Since $N=2^{n-1}>n^4$ for $n\ge781$, the right-hand side above is
positive. The inequality $2^{n-1}>n^4$ holds at $n=781$ and
persists because $2>(1+1/n)^4$ in this range.  Thus $w_n(2n/21)<L_n$ for every $n\ge781$.

\end{document}